\documentclass[11pt]{article}

\usepackage[top=50mm, bottom=50mm, left=50mm, right=50mm]{geometry}
\usepackage{lineno}
\usepackage{indentfirst}
\usepackage{bm}
\usepackage{amssymb}
\usepackage{amsmath}
\usepackage{amsthm}
\usepackage{algorithm}
\usepackage{algorithmicx}
\usepackage{algpseudocode}
\usepackage{appendix}
\usepackage{cases}
\usepackage{epsfig}
\usepackage{graphicx}
\usepackage{graphics}
\usepackage{float}
\usepackage{subfigure}
\usepackage{multirow}
\usepackage{mathrsfs}
\usepackage{color}
\usepackage{fullpage}
\usepackage[normalem]{ulem}
\usepackage{makeidx}
\usepackage{xspace}
\usepackage{wrapfig}
\usepackage{cite}
\usepackage{tabularx}
\usepackage{latexsym}
\usepackage{amsfonts}
\usepackage{epstopdf}
\makeindex

\newtheorem{thm}{Theorem}[section]

\newtheorem{lem}[thm]{Lemma}

\newtheorem{question}[thm]{Question}

\definecolor{darkred}{rgb}{1, 0.1, 0.3}
\definecolor{darkblue}{rgb}{0.1, 0.1, 1}
\definecolor{darkgreen}{rgb}{0,0.6,0.5}

\newcommand {\mm}[1] {\ifmmode{#1}\else{\mbox{\(#1\)}}\fi}

\floatname{algorithm}{\textbf{Algorithm:}}

\makeatletter
\def\thanks#1{\protected@xdef\@thanks{\@thanks
        \protect\footnotetext{#1}}}
\makeatother

\begin{document}

\title{Minimum saturated graphs for unions of cliques
\thanks{$^{*}$ Corresponding author.
\newline\indent\indent\indent\it\small Email address: zhenhe@bjtu.edu.cn.}
}
\author{Wen-Han Zhu, Rong-Xia Hao and Zhen He$^{*}$
\\
{\it\small  School of Mathematics and Statistics, Beijing Jiaotong University, Beijing 100044, P.R. China.}
}

\date{}

\maketitle

\begin{abstract}
Let $H$ be a fixed graph. A graph $G$ is called {\it $H$-saturated} if $H$ is not a subgraph of $G$
but the addition of any missing edge to $G$ results in an $H$-subgraph.
The {\it saturation number} of $H$, denoted $sat(n,H)$,
is the minimum number of edges over all $H$-saturated graphs of order $n$,
and $Sat(n,H)$ denote the family of $H$-saturated graphs with $sat(n,H)$ edges and $n$ vertices.
In this paper, we resolve a conjecture of Chen and Yuan in[Discrete Math. 347(2024)113868] by determining $Sat(n,K_p\cup (t-1)K_q)$ for every $2\le p\le q$ and $t\ge 2$.

\end{abstract}

\section{Introduction}

{\noindent\bf Notation.}  Let $G$ be a graph with vertex set $V(G)$ of order $n$ and edge set $E(G)$.
For any vertex $x\in V(G)$, the vertex subset $\{y\in V(G): xy\in E(G)\}$
is the {\it neighborhood} of $x$ in $G$, denoted by $N_{G}(x)$.
Let $S$ be a vertex subset of $G$, denoted by $N_{S}(x)=N_{G}(x)\cap S$ for any vertex $x\in V(G)$.
Let $N_{G}[x]=N_{G}(x)\cup \{x\}$ be the \emph {closed neighborhood} of $x$.
The {\it degree} of a vertex $x$ in $G$, denoted by $d_{G}(x)$, is the number of edges of
$G$ incident with $x$.
Let $\delta(G)$ and $\Delta(G)$ denote the minimum and maximum degrees of the vertices of $G$.
For $V^{\prime}\subseteq V(G)$,
$G[V^{\prime}]$ is said to be {\it the subgraph of $G$ induced by $V^{\prime}$}
whose vertex set is $V^{\prime}$ and whose edge set consists of all edges of $G$
with both ends in $V^{\prime}$.
For any $S\subseteq V(G)$, denoted by $\overline{S}=V(G)\backslash S$.
Let $A$ and $B$ be subset of $V(G)$ such that $A\cap B=\emptyset$.
Then denoted by $E_{G}[A,B]$ the set of edges of $G$ such that the two vertices incident to
any edge are respectively in $A$ and $B$.

Let $G$ and $H$ be any two graphs with $V(G)\cap V(H)=\emptyset$.
The {\it union} $G\cup H$  of the graphs $G$ and $H$ has $V(G\cup H)=V(G)\cup V(H)$
and $E(G\cup H)=E(G)\cup E(H)$.
Specially, $tH$ is the union of $t$ pairwise disjoint copies of $H$.
The {\it join} of graphs $G$ and $H$, denoted by $G\vee H$, is the graph with $V(G\vee H)=V(G)\cup V(H)$
and $E(G\vee H)=E(G)\cup E(H)\cup \{xy|x\in V(G),y\in V(H)\}$.
Denote a complete graph and an independent set of order $n$ by $K_{n}$ and $I_{n}$.
 We use $[m]$ to denote the set $\{1,2,\dots,m\}$.

For any two graphs $G$ and $H$, we say that $G$ is {\it $H$-free} if there is no subgraph isomorphic to $H$ in $G$.
A graph $G$ is called {\it $H$-saturated} if it is $H$-free, but for any edge $uv\in E(\overline{G})$, the graph $G+uv$ contains a
subgraph isomorphic to $H$, where $\overline{G}$ denotes the complement of $G$.
Let $sat(n,H)=$~min~$\{|E(G)|:|V(G)|=n$ and $G$ is $H$-saturated$\}$
and $Sat(n,H)=\{G:|V(G)|=n,|E(G)|=sat(n,H),$ and $G$ is $H$-saturated$\}$.
We will call that $sat(n,H)$ is the {\it saturation number} of the graph $H$.
Each member of $Sat(n,H)$ is a {\it minimum saturated graph}.

\vskip.1cm
{\noindent\bf Background.} Saturation numbers were first studied by Erd\H{o}s, Hajnal, and Moon \cite{A.Hajnal}, the authors determined $sat(n,K_p)=\binom{n}{2}-\binom{n-p+2}{2}$ and proved that $K_{p-2}\vee I_{n-p+2}$ is the unique minimum $K_p$-saturated graph. 
Since then, the saturation numbers have been determined for various classes of graphs.
K\'{a}szonyi and Tuza \cite{ZsTuza} gave a general upper bound $sat(n,\mathcal{H})=O(n)$ for any family $\mathcal{H}$ of graphs and determined $sat(n,H)$ for $H\in \{K_{1,k},tK_{2},P_{k}\}$.
Saturation number for cycles was studied in \cite{Ollmann,YCChen,YCChen2,ZFuredi,Tuza,YLan}.
Saturation number for complete partite graph was studied in \cite{YCChen3,Pikhurko,Huang,Bohman}.
Vertex-disjoint graphs saturation problem was studied in \cite{ZsTuza,Faudree,FanChen,Cao,ZLv}. See \cite{Currie} for an abundant survey.

There are very few graphs for which $sat(n,H)$ and $Sat(n,H)$ are known exactly. 
In this paper we will be interested in the setting when the forbidden graph are vertex-disjoint cliques. 
In 1986, K\'{a}szonyi and Tuza \cite{ZsTuza} determined $sat(n,tK_2)$ and gave a complete characterization of the minimum $tK_2$-saturated graphs.
In 2007, Faudree, Ferrara, Gould and Jacobson\cite{Faudree} determined the exact value of $sat(n,tK_{p})$ with $t\geq 1,p\geq 3$ and large $n$. 
Meanwhile, they determined $Sat(n,K_{p}\cup K_{q})=\{ K_{p-2}\vee(K_{q+1}\cup I_{n-p-q+1})\}$ with $q\geq p$ and $Sat(n,3K_{p})=\{ K_{p-2}\vee (2K_{p+1}\cup I_{n-3p})\}$. Furthermore, they asked if $Sat(n,tK_{p})=\{K_{p-2}\vee ((t-1)K_{p+1}\cup I_{n-tp-t+3})\}$ holds for all $t\ge 4$.
\begin{question}\cite{Faudree}\label{ques1}
    Is it true that for $t\ge 4$ we have $Sat(n,tK_{p})=\{K_{p-2}\vee ((t-1)K_{p+1}\cup I_{n-tp-t+3})\}$ when $n$ is large enough?
\end{question}
In 2024, Chen and Yuan\cite{FanChen} determined $sat(n,K_p\cup (t-1)K_q)$ for $t\ge 3$ and $2<p\le q$ and $sat(n,K_p\cup K_q\cup K_r)$ for $r\ge p+q$. 
Meanwhile, they determined $Sat(n,K_p\cup K_q\cup K_r)$ for $r\ge p+q$ and conjectured that $Sat(n,K_{p}\cup (t-1)K_q)=\{K_{p-2}\vee ((t-1)K_{q+1}\cup I_{n-t+3-p-(t-1)q})\}$ with $2\le p \le q$, $n>(t-1)(q+1)^2+3(p-2)$ and $t\ge 4$.
In this paper, we prove that this conjecture indeed holds, particularly, set $p=q$, we answer Question \ref{ques1}. Our main result is Theorem \ref{thm1}.
\begin{thm}\label{thm1}
    Suppose $2\le p \le q$, $n>q(q+1)(t-1)+3(p-2)$ and $t\ge 2$. $Sat(n,K_{p}\cup (t-1)K_q)=\{K_{p-2}\vee ((t-1)K_{q+1}\cup I_{n-t+3-p-q(t-1)})\}$.
\end{thm}

The paper consists of three Section. In Section 2, some notations and important properties of $K_{p}\cup (t-1)K_{q}$-saturated graph are introduced.
In Section 3,  the unique $K_{p}\cup (t-1)K_{q}$-saturated graph of minimum size is determined.

\section{Some properties of $K_{p}\cup (t-1)K_{q}$-saturated graph}

Let $t\geq 1$, $2\le p\le q$ and $n\geq p+(t-1)q+t-3$ be fixed integers. Let


$$H(n,p,q,t)\cong K_{p-2}\vee ((t-1)K_{q+1}\cup I_{n-t+3-p-q(t-1)}).$$

In \cite{FanChen,Faudree}, the authors completely determined the saturation number of $K_{p}\cup (t-1)K_{q}$ for $t\geq 1$ and $2\leq p\leq q$ and proved that $H(n,p,q,t)\in Sat(n,K_{p}\cup (t-1)K_{q})$.

\begin{lem}\label{lem1}\cite{FanChen,Faudree}
Let $2\leq p\leq q, t\geq 2$ and $n>q(q+1)(t-1)+3(p-2)$ be integers.
Then 
$$sat(n,K_{p}\cup (t-1)K_{q})=(p-2)(n-p+2)+(t-1)\binom{q+1}{2}+\binom{p-2}{2}.$$
\end{lem}

Next, we list some properties of a graph $G$ with $G\in Sat(n,K_{p}\cup (t-1)K_{q})$ and establish some new properties.

\begin{lem}\label{lem2}\cite{FanChen}
Suppose $2\le p \le q$, $n>q(q+1)(t-1)+3(p-2)$ and $t\ge 2$.
Let $G\in Sat(n,K_{p}\cup (t-1)K_{q})$. Then $\delta(G)=p-2$.
\end{lem}

Let $v$ be a vertex of minimum degree in $G$ and $S=N_G(v)$.
As $G\in Sat(n,K_{p}\cup (t-1)K_{q})$, 
by Lemma \ref{lem2}, we have $\delta(G)=d_{G}(v)=|S|=p-2$.
Notice that $G$ is $K_{p}\cup(t-1)K_{q}$-saturated.
For any vertex $u\in \overline{S\cup\{v\}}$,
there is a subgraph isomorphic to $K_{p}\cup(t-1)K_{q}$ in $G+uv$ such that $uv$ lies in $K_{p}$.
As $|S|=p-2$, $u$ is adjacent to any vertex of $S$.
It implies that $G[S\cap N_{G}(u)]=G[S]\cong K_{p-2}$ and $N_{G}(v)\subseteq N_{G}(u)$ for any $u\not\in N_{G}(v)$.
Since $G[S\cup \{u,v\}]$ is a clique containing $uv$ in $G+uv$,
there are $(t-1)$ disjoint copies of $K_{q}$ in $G\backslash (S\cup \{u,v\})$.
Let $\mathcal{F}_{uv}$ be the family of subgraphs that isomorphic to $(t-1)K_q$ in $G\backslash (S\cup \{u,v\})$. Then the following lemma holds.

\begin{lem}\label{lem3}
For any $u\in \overline{S\cup \{v\}}$
and $F\in \mathcal{F}_{uv}$,
$N_{G}(u)\backslash S \subseteq V(F)$.
\end{lem}

\begin{proof}
Assume that there exists a vertex $w\in N_{G}(u)\backslash S$ and 
$w\not\in V(F)$. Then $F\cup G[S\cup\{u,w\}]$ is a subgraph isomorphic to $K_{p}\cup (t-1)K_{q}$ in $G$, a contradiction.
\end{proof}

\begin{lem}\label{lem4}
Let $H$ be a connected graph, 
$X$ be a vertex subset of $H$.
Then $|E(H)|-|E(H[X])|\geq |V(H)\backslash X|$.
\end{lem}

\begin{proof}
Let $H^{'}$ be the graph obtained by contracting $X$ into a vertex in $H$, then $|E(H^{'})|=|E(H)|-|E(H[X])|$ and $|V(H^{'})|=|V(H)\backslash X|+1$. Since $H^{'}$ is connected, $|E(H)|-|E(H[X])|=|E(H^{'})|\ge |V(H^{'})|-1=|V(H)\backslash X|$.
\end{proof}






\section{Proof of Theorem \ref{thm1}}

\noindent{\it Proof.}
The case for $q=2$ was completely resolved in \cite{ZsTuza}, hence we only need study the case with $q\ge 3$.
Let $G$ be a graph of order $n$ satisfying that $G\in Sat(n,K_{p}\cup (t-1)K_{q})$ and $G\not\cong H(n,p,q,t)$.
Let $v$ be a vertex with minimum degree of $G$ and $S=N_{G}(v)$. Let $F$ be any subgraph isomorphic to $(t-1)K_{q}$ in $G\backslash S$ and  $F=F_{1}\cup F_{2}\cup\dots\cup F_{t-1}$ with $F_{i}\cong K_{q}$ for any $i\in [t-1]$.
Then $E(G[\overline{V(F)\cup S}])=\emptyset$.
Otherwise, suppose that $xy\in E(G[\overline{V(F)\cup S}])$, then $G[S\cup \{x,y\}]\cup F$ is a subgraph isomorphic to
$K_{p}\cup(t-1)K_{q}$ in $G$, which contradicts $G$ being $K_{p}\cup(t-1)K_{q}$-saturated. Let $R_{F}=\{w|w\in \overline{V(F)\cup S}$ and $w$ has at least one neighbor in $F\}$ and
$A_{F}=\{xy|xy\not\in E(F)$ and $x,y\in \overline{S}\}$.
In fact, each edge of $A_{F}$ incident to at least one vertex in $F$ as $E(G[\overline{V(F)\cup S}])=\emptyset$.
Then
\begin{eqnarray}    \label{eq}
|E(G)|&= &|E(G[S])|+|E(F)|+|E_{G}[S,\overline{S}]| +|A_F| \nonumber    \\
&=&\binom{p-2}{2}
+(t-1)\binom{q}{2}
+(p-2)(n-p+2)+|A_{F}|,\nonumber
\end{eqnarray}
By $G\in Sat(n,K_{p}\cup(t-1)K_{q})$ and Lemma \ref{lem1} , $|A_{F}|=(t-1)q$.

Assume that there is $u\in V(F_i)$ such that $N(u)=S\cup V(F_i)$ for $i\in [t-1]$. Let $w$ be any other vertex in $V(F_i)$. Then $G+vw$ contains a subgraph isomorphic to $K_{p}\cup(t-1)K_{q}$ in which $G[S\cup \{v,w\}]$ is a copy of $K_p$ in $G+vw$, hence $u$ cannot lie in this subgraph, a contradiction to Lemma \ref{lem3}. Therefore, we can assume that every vertex $u$ in $F$ has a neighbor $u^{\prime}$ that lies in $R_F$ or $F$ such that $uu^{\prime}$ is not in $E(F)$.

Let $C_{1},C_{2},\dots,C_{h}$ be the components of $G[R_{F}\cup V(F)]-E(F)$,
where $h\geq 1$.
For any $i\in [h]$, the component $C_{i}$ falls into three categories:
there is a cycle in $C_{i}$;
$C_{i}$ is a tree and there exists at least one edge in $C_{i}$ which is incident with a vertex in $R_{F}$;
$C_{i}$ is a tree and any vertex in $C_{i}$ has no neighbour in $R_{F}$.
According to the following lemma, we only need to consider the first two types.

\begin{lem}\label{lem5}\cite{Faudree}
There is no component $C$ in $G[R_{F}\cup V(F)]-E(F)$ such that $C$ is a tree and any vertex of $C$ has no neighbour in $R_{F}$.
\end{lem}

For any $i\in [h]$, 
if $C_{i}$ is a tree,
by Lemma \ref{lem5}, 
$|E(C_{i})|=|V(C_{i})|-1= |V(C_{i})\cap V(F)|+|V(C_{i})\cap R_{F}|-1\geq |V(C_{i})\cap V(F)|$.
If $C_{i}$ contains a cycle,
then $|E(C_{i})|\geq |V(C_{i})|\geq |V(C_{i})\cap V(F)|$.
Therefore, $|E(C_{i})|\geq|V(C_{i})\cap V(F)|$
for any $i\in [h]$.
It implies that $\sum_{i=1}^{h}|E(C_{i})|\geq \sum_{i=1}^{h}|V(C_{i})\cap V(F)|=(t-1)q$.
Notice that $\sum_{i=1}^{h}|E(C_{i})|=|A_{F}|=(t-1)q$.
Then $\sum_{i=1}^{h}|E(C_{i})|=\sum_{i=1}^{h}|V(C_{i})\cap V(F)|$, and so 

\begin{eqnarray}
|E(C_{i})|&=&|V(C_{i})\cap V(F)| ~\text{for any}~i\in [h].
\end{eqnarray}

\noindent{\bf Claim~1.} {\it $(i)$ $|V(C_i)\cap R_F|\le 1$ for any $i\in [h]$.

 $(ii)$ Let $x,y$ and $r$ be any three distinct vertices of $G$ such that $x\in V(F_{i}),y\in V(F_{j})$ and $r\in R_{F}$
for some $i,j\in [t-1]$ and $i\neq j$.
If $xr,yr\in E(G)$, then $xy\not\in E(G)$.

 $(iii)$ For $q\ge 4$, let $X$ be a vertex subset of $F$ and $G[X]\cong K_{q}$. Then there exists $i_1, i_2\in [h]$ such that $V(X) \subseteq V(F_{i_1})\cup V(F_{i_2})$.}

\begin{proof}
$(i)$ If  $|V(C_i)\cap R_F|\ge 2$ for some $i\in [h]$, then $|E(C_{i})|\ge |V(C_{i})|-1\ge |V(C_{i})\cap V(F)|+1$, a contradiction to Equation (1).

$(ii)$ Assume for a contradiction that $xy\in E(G)$, then 
$x,y$ and $r$ in the same component of $C_{i}$ for some $i\in [h]$.
Without loss of generality, let $x,y,r\in V(C_{1})$
and denoted by $X=\{x,y,r\}$
By Lemma \ref{lem4},
$|E(C_{1})|\geq|E(C_{1}[X])|+|V(C_{1})\backslash X|=3+|V(C_{1})|-3\geq|V(C_{1})\cap V(F)|+1,$
a contradiction to Equation (1).

$(iii)$ Suppose there are $i_1,i_2,i_3\in [h]$ such that $a_j=|V(X)\cap V(F_{i_j})| \neq 0$ for $j\in [3]$. Then 
\begin{eqnarray}    \label{eq}
|E(G[X])\setminus E(F)|-|X|&\geq & a_1a_2+a_1(q-a_1-a_2)+a_2(q-a_1-a_2)-q \nonumber    \\
&=&a_1a_2-1+(a_1+a_2-1)(q-a_1-a_2-1).  \nonumber 
\end{eqnarray}
Since $q\ge 4$, we have $a_1a_2-1>0$ or $q-a_1-a_2-1>0$. Thus $|E(G[X])\setminus E(F)|-|X|>0$. Let $C_1$ be the component contains $X$, then by Lemma \ref{lem4} we have $|E(C_1)|\ge |E(C_1[X])|+|V(C_1)\setminus X|>|X|+|V(C_1)\setminus X|=|V(C_1)|\ge |V(C_1)\cap V(F)|$, a contradiction to Equation (1).
\end{proof}

\noindent{\bf Claim~2.} {\it Let $r\in R_{F}$ such that $|N(r)\cap V(F_{i})|\ge 2$ for some $i\in [t-1]$.
Then $V(F_{i})\subseteq N(r)$.}

\begin{proof}
Suppose $x_1,x_2\in V(F_1)\cap N(r)$, let $F_{vx_1}\in \mathcal{F}_{vx_1}$ be $t-1$ copies of $K_q$ in $G\setminus(S\cup \{v,x_1\})$. By Lemma \ref{lem3} and $\{r,x_2\}\subseteq N(x_1)$, we have $\{r,x_2\}\subseteq V(F_{vx_1})$. If $r,x_2$ are in the same copy of $K_q$ in $F_{vx_1}$, then by Claim 1(ii), $G[\{r\}\cup V(F_1)\setminus\{x_1\}]$ forms a copy of $K_q$ in $F_{vx_1}$. Thus $V(F_{i})\subseteq N(r)$. If $r,x_2$ are in different copies of $K_q$ in $F_{vx_1}$. $F_{vx_1}$ is a subgraph isomorphic to $(t-1)K_q$ in $G\setminus S$, $x_1\in R_{F_{vx_1}}$ and $r,x_2$ are in different copies of $K_q$ in $F_{vx_1}$, a contradiction to Claim 1(ii) and we are done. 
\end{proof}


\noindent{\bf Claim~3.} {\it $|N(r)\cap  V(F_i)|\neq 1$ for any $r\in R_{F}$ and $i\in [t-1]$.}

\begin{proof}
Suppose $|N(r)\cap  V(F_i)|= 1$ for some $i\in [t-1]$, say $|N(r)\cap  V(F_1)|= 1$. Let $\{f_1\}=N(r)\cap  V(F_1)$ and $F_{vf_1}\in \mathcal{F}_{vf_1}$ be $t-1$ copies of $K_q$ in $G\setminus (S\cup \{v,f_1\})$. Then by Lemma \ref{lem3}, $r$ is in a copy of $K_q$ in $F_{vf_1}$. By Claim 1(ii), $r$ together with $q-1$ vertices in $V(F_{i})(i\neq 1)$, say $i=2$, form a copy of $K_q$ in $F_{vf_1}$, let $f_2$ be the vertex in $V(F_{2})$ but not in the $K_q$ containing $r$ in $F_{vf_1}$. If $f_2\notin F_{vf_1}$, then $(F_{vf_1}\setminus (V(F_2)\cup \{r\}))\cup F_2\cup G[S\cup\{v,f_{1}\}]$ is a subgraph isomorphic to $K_p\cup(t-1)K_q$ in $G$, a contradiction. Thus $f_2\in F_{vf_1}$. Let $F^{\prime}=(F\setminus F_2)\cup G[(V(F_2)\cup\{r\})\setminus \{f_2\}]$. Then $F^{'}$ is a subgraph isomorphic to $(t-1)K_q$ in $G\setminus S$. By Claim 1(ii), $f_2$ together with $q-1$ vertices in $V(F_{i})(i\neq 2)$ form a copy of $K_q$ in $F_{vf_1}$. If $i\neq 1$, say $i=3$, let $f_3$ be the vertex in $V(F_{3})$ but not in the $K_q$ containing $f_2$ in $F_{vf_1}$. By Claim 2, $f_2f_3\in E(G)$. Then we continue this iterative process until $f_j$ together with $q-1$ vertices in $V(F_{1})\setminus \{f_1\}$ form a copy of $K_q$ in $F_{vf_1}$ (see Figure \ref{Fig.2}). Then $\{r,f_2,f_3,...,f_j,f_1\}$ contains a cycle in $G[R_F\cup V(F)]-E(F)$, a contradiction to Equation (1).
\end{proof}

\begin{figure}[htbp]
\begin{center}
\scalebox{0.5}[0.5]{\includegraphics{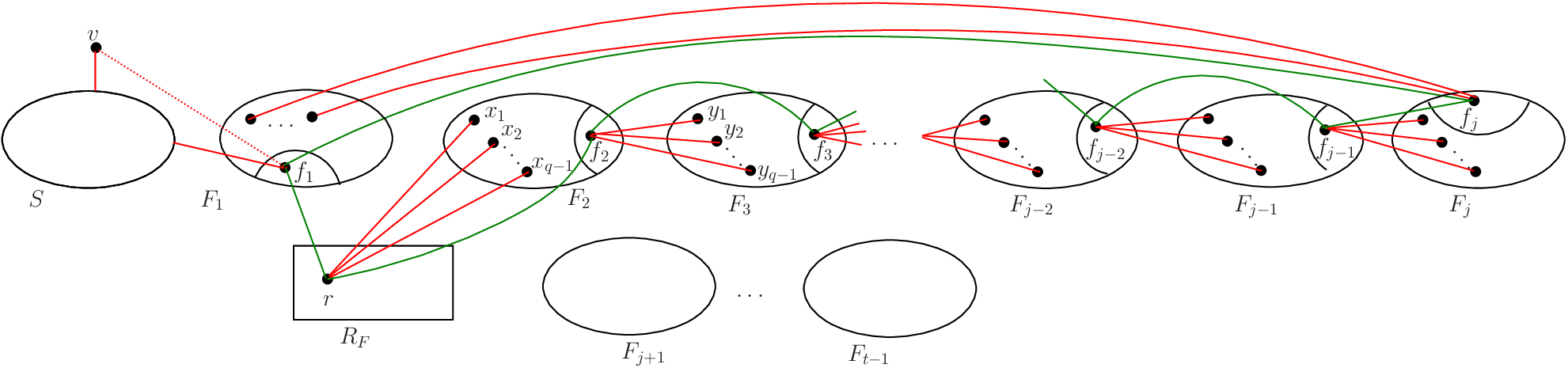}}
\caption{Illustration for Claim 3 }\label{Fig.2}
\end{center}
\end{figure}

\noindent{\bf Claim~4.} {\it For any $r\in R_{F}$, $N_{F}(r)=V(F_{i})$ for some $i\in [t-1]$.}

\begin{proof}
Suppose that there are $i,j\in [t-1]$ such that $N_{F}(r)\cap V(F_{i})\neq \emptyset$ and $N_{F}(r)\cap V(F_{j})\neq \emptyset$, say $i=1,j=2$. By Claims 2 and 3, $V(F_{1})\cup V(F_{2})\subseteq N_{F}(r)$. If there exists a vertex $\alpha\in V(F_{k})$ for some $k\in \{3,4,\dots,t-1\}$ such that $r\alpha\in E(G)$, then $r$ is adjacent to all vertices in $V(F_{k})$.

Let $a$ be an integer with $2\leq a\leq t-1$. Without loss of generality, assume that $|E_{G}[r,V(F_{i})]|=q$ for any $i\in [a]$ and $E_{G}[r,V(F_{j})]=\emptyset$ for any $j\in \{a+1,\dots,t-1\}$. 
By Claim 1 $(i)$, $E_{G}[\bigcup_{i=1}^{a}V(F_{i}),$
$R_{F}\backslash \{r\}]=\emptyset$.
If there exist some $i\in[a]$ and $j\in \{a+1,\dots,t-1\}$
such that $E_{G}[V(F_{i}),V(F_{j})]\neq \emptyset$,
without loss of generality, assume that $x_{1}y_{1}\in E(G)$ with $x_{1}\in V(F_{1})$
and $y_{1}\in V(F_{a+1})$.
Let $M_{1}=G[\{r\}\cup V(F_{1})\backslash \{x_{1}\}]$.
Clearly, $M_{1}\cong K_{q}$.
Let $F^{\prime}=M_{1}\cup (\bigcup_{i=2}^{t-1}F_{i})$
and $R_{F^{\prime}}=\{x_{1}\}\cup R_{F}\backslash\{r\}$.
Then $F^{\prime}\cong (t-1)K_{q}$ and $E(G[R_{F^{\prime}}])=\emptyset$.
Then by Claims 2 and 3, $x_{1}$ is adjacent to all vertices in $F_{a+1}$.
By Claim 1 $(i)$, $E_{G}[V(F_{a+1}),R_{F^{\prime}}\backslash \{x_{1}\}]=\emptyset$.
That is $E_{G}[V(F_{a+1}),R_{F}]=\emptyset$.

Repeat the above process, without loss of generality,
 we obtain a maximum connected graph $G_{1}$ in $G\setminus S$
satisfying that $G[\{r\}\cup (\bigcup_{i=1}^{b}V(F_{i}))]$ is a subgraph of $G_{1}$,
and $E_{G}[V(G_{1}),\bigcup_{j=b+1}^{t-1}V(F_{j})]=\emptyset$, where $b$ is an integer with $a\leq b\leq t-1$
(see Figure \ref{Fig.3}).
We can also get $E_{G}[V(F_{j}),R_{F}]=\emptyset$
by a similar argument as we used to get $E_{G}[V(F_{a+1}),R_{F}]=\emptyset$ for any $j\in \{a+2,\dots,b\}$.
Therefore $V(G_{1})\cap R_{F}=\{r\}$.
By Lemma \ref{lem4} and Equation (1), there exist two vertices, say $f_{1}$ and $f_{2}$, in $\bigcup_{i=1}^{b}V(F_i)$ such that $f_{1}f_{2}\not\in E(G)$.
As $G$ is $K_{p}\cup (t-1)K_{q}$-saturated, 
there is a subgraph $K_{p}\cup (t-1)K_{q}$ in $G+f_{1}f_{2}$
that is induced by the vertex set $S\cup R_{F}\cup V(F)$.
Since $|V(G_{1})\cup S|=bq+p-1$,
there are at most $b$ disjoint copies of $K_{q}$ containing vertices in $V(G_{1})$ in $G+f_1f_2$.
Therefore, there is a subgraph isomorphic to $K_{p}\cup (t-b-1)K_{q}$ in 
$G[\bigcup_{i=b+1}^{t-1}V(F_{i})\cup (R_{F}\backslash \{r\})\cup S]$,
say $M_{1}$.
Thus $\bigcup_{i=1}^{b}F_{i}\cup M_{1}$ is a subgraph isomorphic to $K_{p}\cup (t-1)K_{q}$ of $G$,
a contradiction to $G$ being $K_{p}\cup (t-1)K_{q}$-saturated.  
\end{proof}



\begin{figure}[htbp]
\begin{center}
\scalebox{0.5}[0.5]{\includegraphics{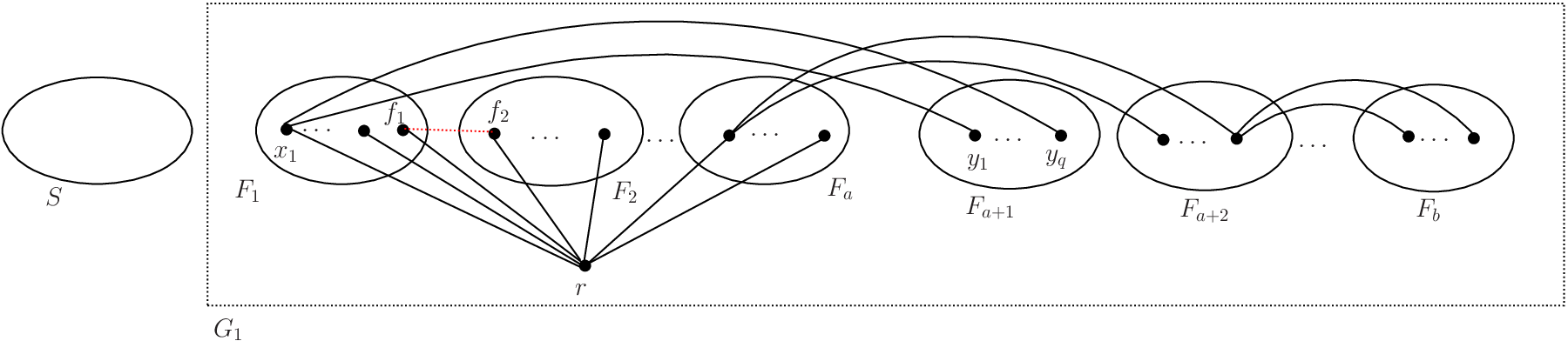}}
\caption{Illustration for Claim 4}\label{Fig.3}
\end{center}
\end{figure}

Let $r\in R_{F}$, by Claim 4, we can assume $N(r)=S\cup V(F_1)$. Then for any $u\in V(F_1)$, $G[(V(F_1)\cup\{r\})\setminus\{u\}]\cup (\bigcup_{i=2}^{t-1}F_{i})$ are $t-1$ copies of $K_q$ in $G\setminus S$. Hence by Claim 4, $N(u)=(S\cup V(F_1)\cup\{r\})\setminus \{u\}$. Therefore, we obtain an independent clique $G[V(F_{1})\cup \{r\}]$ with order $q+1$ of $G$ subject to
$\bigcup_{i=2}^{t-1}F_{i}\cup R_{F}\backslash \{r\}$. Then consider $G+vw$ for any $w\in \bigcup_{i=2}^{t-1}V(F_{i})$, 
there exists $r_1\in R_{F}$ such that $N(r_1)\cap (\bigcup_{i=2}^{t-1}V(F_{i}))\neq \emptyset$, say $N(r_1)\cap V(F_2)\neq \emptyset$. Therefore, we obtain an independent clique $G[V(F_{2})\cup \{r_1\}]$ with order $q+1$ of $G$ subject to
$F_1\cup (\bigcup_{i=3}^{t-1}F_{i})\cup R_{F}\backslash \{r_1\}$. Repeat the above proof, we can obtain that $G[V(F)\cup R_{F}]=(t-1)K_{q}$.
Recall that $G[S]\cong K_{p-2}$.
Therefore, $G\cong H(n,p,q,t)$, which contradicts to our assumption
 $G\not\cong H(n,p,q,t)$.
Hence, $Sat(n,K_{p}\cup (t-1)K_{q})=\{K_{p-2}\vee ((t-1)K_{q+1}\cup I_{n-t+3-p-q(t-1)})\}$.
$\hfill\blacksquare$

\section*{Declaration of competing interest}

This paper does not have any conflicts to disclose.



 \section*{Acknowledgements}
 
The project is supported by Beijing Natural Science Foundation (no. 1244047), China Postdoctoral Science Foundation (no. 2023M740207), the National Natural Science Foundation of China (no. 12331013, 12161141005) and the Talent Fund of Beijing Jiaotong University (no. 2024-003).

\def\polhk#1{\setbox0=\hbox{#1}{\ooalign{\hidewidth
\lower1.5ex\hbox{`}\hidewidth\crcr\unhbox0}}}

\end{document}